\documentclass[12pt,leqno]{article}

\usepackage[english]{babel}
\usepackage{amsmath}
\usepackage{amssymb}
\usepackage{amsthm}
\usepackage{graphicx,xcolor}

\title{Error bound for the asymptotic expansion of the Hartman-Watson integral}
\author{Dan Pirjol}
\date{Stevens Institute of Technology\\
Hoboken, NJ 07030\\
\today}

\newtheorem{proposition}{Proposition}
\newtheorem{remark}{Remark}

\begin{document}
\maketitle

\begin{abstract}
This note gives a bound on the error of the leading term in the $t\to 0$ asymptotic expansion 
of the Hartman-Watson distribution $\theta(r,t)$ in the regime $rt=\rho$ constant.
The leading order term has the form $\theta(\rho/t,t)=\frac{1}{2\pi t}e^{-\frac{1}{t} (F(\rho)-\pi^2/2)} G(\rho) (1 + \vartheta(t,\rho))$, where the error term is bounded uniformly 
over $\rho$ as 
$|\vartheta(t,\rho)|\leq \frac{1}{70}t$.
\end{abstract}

\section{Introduction}

The Hartman-Watson distribution \cite{Hartman} appears in several problems of applied probability and financial engineering. Most notably this distribution determines the joint distribution of the
time-integral of a geometric Brownian motion and its terminal value \cite{Yor}. 
The precise numerical evaluation of this distribution is of interest for many applications, see for example \cite{Aimi,Barrieu,Buonaguidi}.
 
This distribution is expressed in terms of the Hartman-Watson integral $\theta(r,t)$, defined as
\begin{equation}\label{thetadef}
\theta(r,t) = \frac{r}{\sqrt{2\pi^3 t}} e^{\frac{\pi^2}{2t}} \int_0^\infty e^{-\frac{1}{2t} \xi^2 - r \cosh \xi} \sinh \xi \sin \frac{\pi \xi}{t} d\xi \,.
\end{equation}
The numerical evaluation of this integral for small $t\ll 1$ requires very high accuracy in intermediate steps, due to the fast oscillating factor in the integrand $\sin(\frac{\pi \xi}{t})$ and to the smallness of the integral which is multiplied with the large exponential factor \cite{Barrieu,Boyle}. For this reason, the use of analytical expansions for numerical evaluation in this regime has been proposed as a more convenient alternative \cite{Barrieu,Gerhold,HWpaper}. 

An asymptotic expansion of this integral was proposed in \cite{HWpaper} in the limit $t\to 0$ at fixed $\rho=r t$. Proposition 1 in \cite{HWpaper} gives this expansion as
\begin{equation}\label{HWexp}
\theta(\rho/t,t) = \frac{1}{2\pi t} e^{-\frac{1}{t} (F(\rho) - \frac{\pi^2}{2})}
\Big( G(\rho) + G_1(\rho) t + O(t^2)\Big)\,,\quad (t\to 0)
\end{equation}
where the functions $F(\rho),G(\rho), G_1(\rho)$ are known in closed form. 
%These functions can be approximated as power series in $\log(1/\rho)$ which are convenient for numerical evaluation 
%\begin{align}
%F(\rho) & =\frac{\pi^2}{2}-1+\log(1/\rho) +\log^2(1/\rho) + O(\log^3(1/\rho)) \\
%G(\rho) &=\sqrt3 \Big(1 + \frac15 \log(1/\rho) - \frac{1}{70}\log^2(1/\rho) + O(\log^3(1/\rho) \Big) \,.
%\end{align}
%The first 10 coefficients in these expansions are given in Table 2 of \cite{JCAM}. These series expansions converge for $|\log(1/\rho) | <  3.49295$, see Proposition 4.2 in \cite{JCAM}. Outside of this range, analytical tail asymptotics for $\rho\to 0$ and $\rho\to \infty$ are available \cite{HWpaper} and can be used for a precise approximation.

A simple approximation for $\theta(\rho/t,t)$ is obtained by truncating the expansion \eqref{HWexp} to the first term, and can be written as
\begin{equation}\label{varthetadef}
\theta(\rho/t,t) = \frac{1}{2\pi t} e^{-\frac{1}{t} (F(\rho) - \frac{\pi^2}{2})}
G(\rho)(1 + \vartheta(t,\rho) )
\end{equation}
with $\vartheta(t,\rho)$ an error term. This has been used for numerical pricing of Asian options in \cite{JCAM} and for deriving subleading corrections to the short-maturity asymptotics of Asian options in the Black-Scholes model \cite{subleading}. 
(The leading order term follows from Large Deviations theory and was computed in \cite{PZAsian}.) The exponential factor in \eqref{varthetadef} determines the short maturity asymptotics of European and VIX options in local-stochastic volatility models with 
geometric Brownian motion stochastic volatility \cite{VIXpaper}.

Using a combination of analytical and numerical estimates for the integrand appearing in the asymptotic expansion we give in this note an upper bound on the error term
\begin{equation}
|\vartheta(t,\rho) | \leq \frac{1}{70} t \,.
\end{equation}
This bound is the main result of this note. In Remark \ref{rmk:strongerbound} we give 
also an improved error bound, which remains bounded as $t\to \infty$.
%%%%%%%%%%%%%%%%%%%%%%%%%%%%%%%%%%%%%%%%

\section{Saddle point expansion for $\theta(r,t)$}
\label{sec:2}

We summarize in this section a few steps in the derivation of the asymptotic expansion (\ref{HWexp}) which will be required for the proof of the error bound. 
The asymptotic expansion (\ref{HWexp}) is obtained by 
expressing the integral in (\ref{thetadef}) with $r=\rho/t$ in terms of the integral
\begin{equation}\label{Idef}
I(\rho,t) := \int_{-\infty}^\infty e^{-\frac{1}{t} h(\xi)} \sinh \xi d\xi
\end{equation}
with
\begin{equation}\label{hdef}
h(\xi ) = \frac12\xi^2 + \rho \cosh \xi - i\pi \xi
\end{equation}
The asymptotics of $I(\rho,t)$ as $t\to 0$ of this integral can be computed using the
saddle point method, see for example 
Sec.~4.6 in Erd\'elyi \cite{Erdelyi} and Sec.~4.7 of Olver \cite{Olver}.

For the application of this method, the integration contour in (\ref{Idef}) is deformed from the real axis such that it runs through appropriate saddle points of $h(\xi)$ and along steepest descent paths, along which $\Im h(\xi) = 0$. The position of the saddle points and the choice of
the integration contours depend on $\rho$, as follows.
%For more details see the proof of Proposition 1 in \cite{HWpaper}.

i) For $0<\rho<1$ the integration contour is shown in the left plot of Fig.~\ref{Fig:contours}.
It passes through the saddle points at $B: \xi_B=-x_1+i\pi$ and $A: \xi_A=x_1 + i \pi$
where $x_1$ is the solution of the equation
\begin{equation}\label{eqx1}
\rho \frac{\sinh x_1}{x_1} = 1
\end{equation}

ii) For $\rho > 1$ the integration contour runs as in the middle plot in Fig.~\ref{Fig:contours},
and passes through the saddle point $S$ at $\xi_S =i y_1$,
where $y_1$ is the solution of the equation 
\begin{equation}\label{eqy1}
y_1 + \rho\sin y_1 = \pi\,.
\end{equation}

iii) $\rho=1$. The integration contour is shown in the right plot of Fig.~\ref{Fig:contours}.
This passes through the fourth order\footnote{The first non-zero derivative
of $h(\xi)$ at this point is the fourth order derivative.} saddle point at $S: \xi_S = i\pi$. 
\vspace{0.2cm}

\begin{figure}
    \centering
   \includegraphics[width=1.5in]{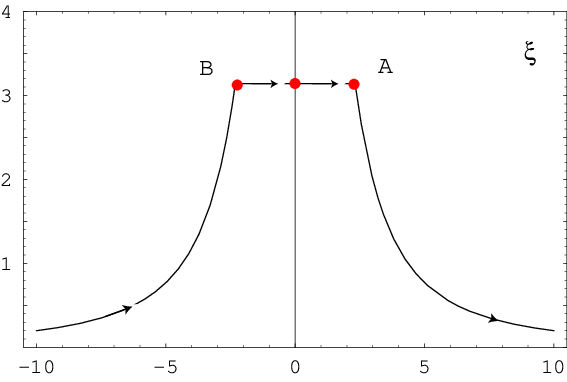}
   \includegraphics[width=1.5in]{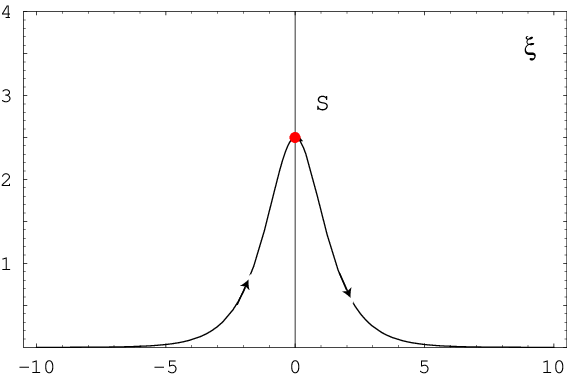}
   \includegraphics[width=1.5in]{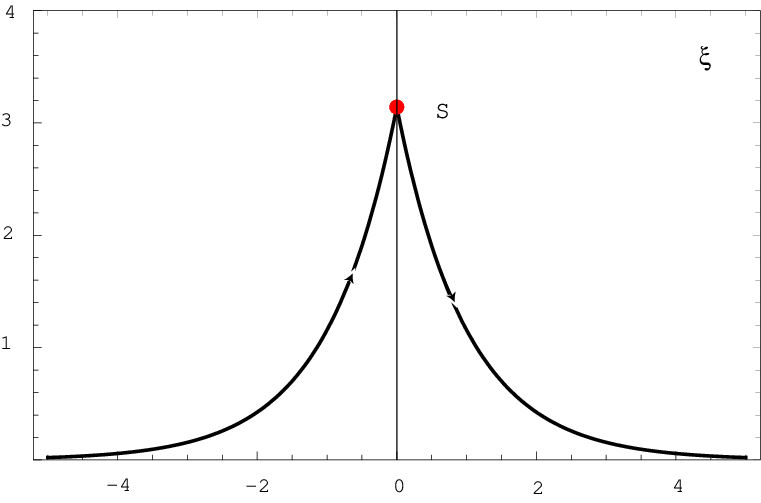}
    \caption{ Integration contours for $I(\rho,t)$ in the $\xi$ complex plane
for the application of the asymptotic expansion. 
The red dots show the saddle points. 
Left: contour for $0<\rho<1$. The contour passes through the saddle points 
$B(\xi=-x_1+i\pi)$ and $A(\xi=x_1+i\pi)$.
Middle: contour for $\rho>1$. 
The contour passes through the saddle point $S(\xi=i y_1)$.
Right: the contour for $\rho=1$, passing through the saddle point $S(\xi=i\pi)$.}
\label{Fig:contours}
 \end{figure}

For all cases, the contour integrals giving the Hartman-Watson integral can be expressed as the imaginary part of an integral
\begin{equation}\label{HWint}
\theta(\rho/t,t) = \frac{\rho}{\sqrt{2\pi^3 t^3}} e^{\frac{\pi^2}{2t}} e^{-\frac{1}{t} h(X)}
\Im \int_0^\infty e^{-\frac{1}{t}\tau} g(\xi(\tau),\rho) d\tau
\end{equation}
where $X$ is one of the saddle points, distinct for each case: 
i) $X=A$ for $0< \rho<1$, ii) $X=S$ for $\rho>1$ and iii) $X=S$ for $\rho=1$.
The function $g(\xi(\tau),\rho)$ is 
\begin{equation}
g(\xi,\rho) = \frac{\sinh\xi}{\xi + \rho \sinh\xi - i\pi} 
\end{equation}
taken along the steepest descent path $\xi(\tau): [X,\infty)$ starting at the saddle point $X$ and extending to $+\infty$. The real variable $\tau$ along the path is defined by $\tau = h(\xi) - h(X)$ where $h(\xi)$ is defined in (\ref{hdef}).
We will denote for simplicity $g(\tau,\rho) := g(\xi(\tau ),\rho)$.

The integrand is expanded as
\begin{equation}\label{imgexp}
\Im  g(\tau,\rho) = g_0(\rho) \frac{1}{\sqrt{\tau}} + g_2(\rho) \sqrt{\tau} + O(\tau^{3/2})
\end{equation}
The coefficient $g_0(\rho)$ is given explicitly as follows.
\begin{equation}
g_0(\rho) = \left\{
\begin{array}{cc}
\frac{\sinh x_1}{\sqrt{2(\rho \cosh x_1 - 1)}} & \,, 0 < \rho < 1 \\
\sqrt{\frac32} & \,, \rho = 1 \\
\frac{\sin y_1}{\sqrt{2(\rho \cos y_1 + 1)}} & \,, \rho > 0 
\end{array}
\right.
\end{equation}
where $x_1$ is the solution of the equation $\rho \frac{\sinh x_1}{x_1} = 1$ and $y_1$ is the solution of the equation $y_1 + \rho \sin y_1 = \pi$.

Substituting \eqref{imgexp} into the integral (\ref{HWint}) and integrating term by term by Watson's lemma
gives
\begin{equation}
\theta(\rho/t,t) = \frac{\rho}{\sqrt2 \pi t} e^{-\frac{1}{t}(F(\rho) - \frac{\pi^2}{2}) }
\Big(g_0(\rho) +
\frac12 t g_2(\rho) + O(t^2) \Big)
\end{equation}
The leading term has the form shown in Proposition 1 of \cite{HWpaper} by identifying
$G(\rho) = \sqrt2\rho g_0(\rho)$.

%%%%%%%%%%%%%%%%%%%%%%%%%%%%%%%%%%%%%%%%%%%%%%%%%%%%%%%%%%%%%%%%%%%%%

\section{Error bound}
\label{sec:3}

We study here the error introduced by keeping only the leading order term in the expansion 
\eqref{imgexp} of $\Im g(\tau,\rho)$ in the integral in (\ref{HWint}). 
The integral can be written as
\begin{align}\label{intexp}
\Im \int_0^\infty e^{-\tau/t} g(\tau) d\tau &= 
\sqrt{\pi t} \, g_0(\rho) ( 1 + \vartheta(t,\rho) )
\end{align}
where $\vartheta(t,\rho)$ is an error term. Substituting into \eqref{HWint} this yields the representation \eqref{varthetadef} of the function $\theta(\rho/t,t)$.

Define the error of the leading order term in the expansion (\ref{imgexp}) in terms of a function $\delta(\tau,\rho)$
\begin{equation}\label{deltadef}
\Im g(\tau,\rho) =  g_0(\rho) \frac{1}{\sqrt{\tau}} ( 1 + \delta(\tau,\rho))\,.
\end{equation}

The exact results $\rho=1$ presented in the next section
and numerical tests for general $\rho > 0$ in Sec.~\ref{sec:num} suggest 
that $\delta(\tau,\rho)$ is bounded as
\begin{equation}\label{delbound}
|\delta(\tau,\rho) | < \frac{1}{35} \tau\,,\quad \tau > 0
\end{equation}
uniformly over $\rho$.

This bound can be used to derive an upper bound on the error function $\vartheta(t,\rho)$
defined in (\ref{intexp}).

\begin{proposition}\label{prop:main}
Assume that the bound (\ref{delbound}) holds. 
Then the error $\vartheta(t,\rho)$ in (\ref{intexp}) is bounded from above as
\begin{equation}\label{varthetabound}
|\vartheta(t,\rho)| \leq  \frac{1}{70} t\,.
\end{equation}
\end{proposition}

\begin{proof}

We have
\begin{align}
&\Big|\int_0^\infty e^{-\tau/t} ( \Im  g(\tau) - g_0(\rho)\frac{1}{\sqrt{\tau}}) 
d\tau \Big| \leq 
\int_0^\infty e^{-\tau/t} | \Im  g(\tau) - g_0(\rho)\frac{1}{\sqrt{\tau}} | 
d\tau  \\
&\leq 
\int_0^\infty e^{-\tau/t} g_0(\rho) |\delta(\tau,\rho) | \frac{d\tau}{\sqrt{\tau}} 
\leq \frac{1}{35}
\int_0^\infty e^{-\tau/t} g_0(\rho) \sqrt{\tau} d\tau \nonumber \\
&= \frac{1}{70} t g_0(\rho) \sqrt{\pi t} \,.
\nonumber
\end{align}
where we used the bound (\ref{delbound}) in the last step. This is equivalent with the bound
(\ref{varthetabound}) on $|\vartheta(t,\rho)|$.
\end{proof}

\begin{figure}
    \centering
 \includegraphics[width=3.5in]{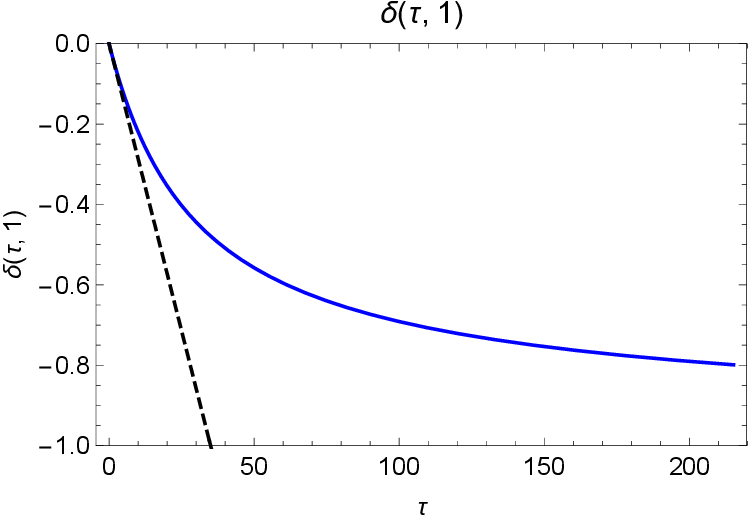}
    \caption{ Plot of $\delta(\tau,1)$ for $\rho=1$. The dashed line is $-\frac{1}{35} \tau$. }
\label{Fig:delta:rho1}
 \end{figure}
 
 \subsection{Some exact results for $\rho=1$}
 
We give here a few exact results about the function $\delta(\tau,1)$. 

 \begin{proposition}
 We have
 \begin{align}
 \lim_{\tau\to 0} \delta(\tau,1) &= 0 \\
\lim_{\tau\to 0} \delta'(\tau,1) &= -\frac{1}{35} \\
\lim_{\tau\to 0} \delta''(\tau,1) &= \frac{7}{4,125} \,.
\end{align}
\end{proposition}

This shows that for sufficiently small $\tau$, the function $\delta(\tau,1)$ is decreasing and convex, and its slope is bounded in absolute value by $1/35$. These features of 
$\delta(\tau,1)$ are observed in 
Figure \ref{Fig:delta:rho1} which shows the numerical evaluation of this function.

\begin{proof}
Taking $\rho=1$ we have
\begin{equation}
\tau = h(\xi) - h(i\pi) = \frac12 \xi^2 + \cosh \xi - i\pi \xi - \frac{\pi^2}{2}-1\,.
\end{equation}
This gives an equation for $\xi$ along the steepest descent path, which can be solved to find 
$\xi(\tau)$.
The form of this equation simplifies by introducing $\zeta = \xi - i\pi$, the distance from a point on the path to the saddle point at $i\pi$. Expressed in terms of $\zeta$, we have
\begin{equation}\label{zetaeq}
\tau = \frac12 \zeta^2 - \cosh \zeta + 1
\end{equation}
and 
\begin{equation}\label{gzeta}
g(\zeta,1) = \frac{\sinh \zeta}{\sinh \zeta - \zeta}\,.
\end{equation}
We denoted for simplicity $g(\zeta,1) = g(\xi(\zeta),1)$.

The equation (\ref{zetaeq}) for $\zeta$ can be solved in a series expansion in $\tau$. 
Substituting into $g(\zeta,1)$ given in (\ref{gzeta}) and expanding in $\tau$ gives an explicit expansion for $\delta(\tau,1)$. The first two terms in this expansion are 
\begin{equation}\label{deltasol}
\delta(\tau,1) = - \frac{1}{35} \tau + \frac{7}{8,250} \tau^2 + O(\tau^3)
\end{equation}
The stated results follow immediately from the coefficients of this expansion.

For completeness we give a few steps in the derivation of (\ref{deltasol}).
Inversion of (\ref{zetaeq}) around $\tau=0$ gives
\begin{equation}
\zeta^2(\tau) = 2\sqrt6 \sqrt{-\tau} + \frac25 \tau + \frac{2}{35} \sqrt{\frac23} (-\tau)^{3/2} + O(\tau^2)\,.
\end{equation}
Substituting into $g(\xi)$ given in (\ref{gzeta}) and expanding in $\tau$ gives
\begin{equation}
g(\tau) = \sqrt{\frac32} \frac{1}{\sqrt{-\tau}} +\frac45 + \frac{1}{35} \sqrt{\frac32} \sqrt{-\tau} + O(\tau^{3/2})\,.
\end{equation}
We are interested in the solution of (\ref{zetaeq}) corresponding to $\zeta$ in the fourth quadrant. 
This is obtained by taking $\sqrt{-\tau} = -i \sqrt{\tau}$, which gives
\begin{equation}\label{img2}
\Im g(\tau,1) = \sqrt{\frac32} \frac{1}{\sqrt{\tau}} - \frac{1}{35} \sqrt{\frac32} \sqrt{\tau} 
+ \frac{7}{2,750 \sqrt6} \tau^{\frac32} + O(\tau^{5/2})
\end{equation}
Finally we have
\begin{equation}
\delta(\tau,1) =-1 + \sqrt{\frac23} \sqrt{\tau} \Im  g(\tau,1) 
\end{equation}
Substituting here \eqref{img2} gives the series (\ref{deltasol}).

\end{proof}

Numerical study of the series \eqref{img2} to higher orders suggests that this is an alternating series.
The first six terms of this series are
\begin{align}
& \Im g(\tau,1) = \sqrt{\frac32} \frac{1}{\sqrt{\tau}} - \frac{1}{35} \sqrt{\frac32} \sqrt{\tau} 
+ \frac{7}{2,750 \sqrt6} \tau^{\frac32} - \frac{44,081}{656,906,250 \sqrt6} \tau^{\frac52} \\
 & + \frac{1,495,665,023}{1,039,685,521,875,000\sqrt6} \tau^{\frac72} -
 \frac{96,439,937,879}{5,734,608,285,656,250,000\sqrt6} \tau^{\frac92}+ O(\tau^{\frac{11}{2}})\,. \nonumber
\end{align}
Substituting into \eqref{HWint} and integrating over $\tau$, the alternating property is preserved.
This gives 
\begin{align}\label{HWexprho1}
\theta(1/t, t) &= \frac{\sqrt3}{2\pi t} e^{1/t} \Big(1 - \frac{1}{70} t + \frac{7}{11,000} t^2 - \frac{44,081}{1,051,050,000} t^3 \\
&+ \frac{1,495,665,023}{475,284,810,000,000} t^4 - 
\frac{96,439,937,879}{582,563,381,400,000,000} t^5 + O(t^6) \nonumber \,.
\end{align}
The truncation error of such a series at any finite order is 
bounded by the first neglected term. 

%%%%%%%%%%%%%%%%%%%%%%%%%%%%%%%
Next we prove also a result for the $\tau \to \infty$ asymptotics of $\delta(\tau,1)$.

\begin{proposition}
We have 
\begin{equation}\label{largetau}
\delta(\tau,1) = -1 + \pi \sqrt{\frac{2}{3\tau}} + O(\tau^{-3/2} \log(2\tau)) \,,\quad (\tau\to \infty)\,.
\end{equation}
\end{proposition}

\begin{proof}
Asymptotic inversion of the equation \eqref{zetaeq} gives 
\begin{equation}
\zeta = \log(-2\tau) + \frac{1}{2\tau}(\log^2(-2\tau)+3) + O(\tau^{-2})\,,\quad (\tau \to \infty)
\end{equation}
Substituting into $g(\zeta,1)$ gives
\begin{equation}
g(\tau,1) = 1 - \frac{1}{\tau} \log (-2\tau) + O(\tau^{-2} \log^2(-2\tau))
\end{equation}
Taking the imaginary part gives
\begin{equation}
\Im g(\tau,1) = \frac{\pi}{\tau} + O(\tau^{-2} \pi \log(2\tau))\,.
\end{equation}
Expressed in terms of $\delta(\tau,1)$ this yields \eqref{largetau}.

\end{proof}

This proves that $\delta(\tau,1)$ approaches $-1$ from above as $\tau \to \infty$,
which agrees with the numerical evaluation of this function in Figure \ref{Fig:delta:rho1}.

%%%%%%%%%%%%%%%%%%%%%%%%%%%%%%%%%%%
\begin{figure}
    \centering
   \includegraphics[width=2.5in]{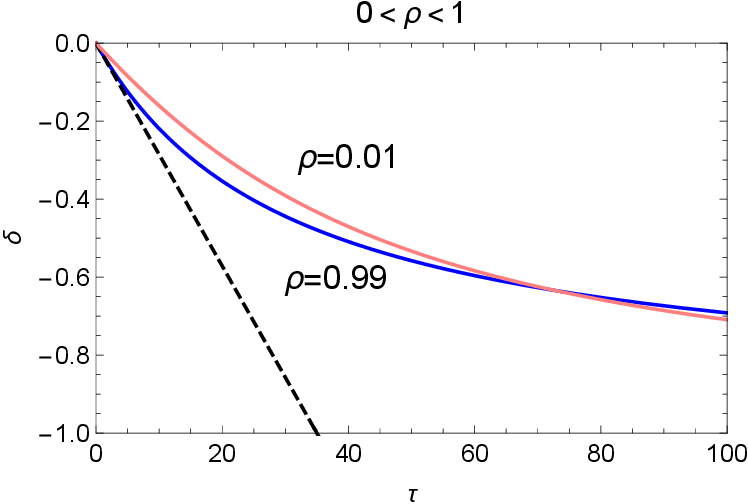}
   \includegraphics[width=2.5in]{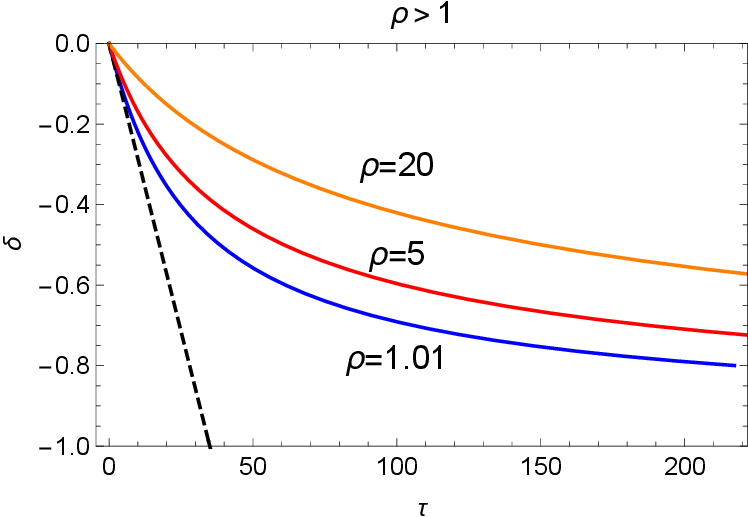}
    \caption{ Plot of $\delta(\tau)$ vs $\tau$ for several values of $\rho$.
    Left:  two extreme values of $\rho$ in the $[0,1]$ interval.
    Right: several values of $\rho$ larger than 1. The dashed line shows the bound 
    $-\frac{1}{35}\tau$.}
\label{Fig:delta}
 \end{figure}

%%%%%%%%%%%%%%%%%%%%%%%%%%%%%%%%%%%
\subsection{Bound on $\delta(\tau,\rho)$ for general $\rho$}
\label{sec:num}

We evaluated numerically the error term $\delta(\tau,\rho)$ defined in (\ref{deltadef}) for several values of $\rho$.  Figure \ref{Fig:delta} shows this function for several values of $\rho$ in $0<\rho<1$ (left) and $\rho>1$ (right). The shape of these plots is similar to that of $\delta(\tau,1)$ in Figure \ref{Fig:delta:rho1}. In particular, we note that $\delta'(0,\rho)$ is negative for all 
$\rho$ and is bounded in absolute value by $\frac{1}{35}$ for all $\rho>0$. 
This is seen more explicitly in 
Figure \ref{Fig:deltaprime} which shows the plot of $\delta'(0,\rho)$ for $\rho \leq 10$. This plot shows that $\delta'(0,\rho)$ reaches its minimum at $\rho=1$, where it takes the 
value $-\frac{1}{35}$.
 
These numerical experiments suggest the following properties of the error function $\delta(\tau,\rho)$ for general $\rho$.

i) $\lim_{\tau\to 0} \delta(\tau,\rho)=0$ for all $\rho >0$.

ii) $\delta(\tau,\rho) < 0$ is negative for all $\tau>0$

iii)  $\delta(\tau,\rho)$ is monotonically decreasing and approaches $-1$ from above as $\tau \to \infty$.

iii) $|\delta(\tau,\rho) | \leq 1$. 

iv) $\delta(\tau,\rho)$ is bounded in absolute value for all $\rho>0$ as 
\begin{equation}\label{dbound}
|\delta(\tau, \rho)| \leq \frac{1}{35} \tau\,,\quad \tau \geq 0\,.
\end{equation}

%%%%%%%%%%%%%%%%%%%%%%%%%%%%%%%%%%%%%%
\begin{figure}
    \centering
   \includegraphics[width=3.5in]{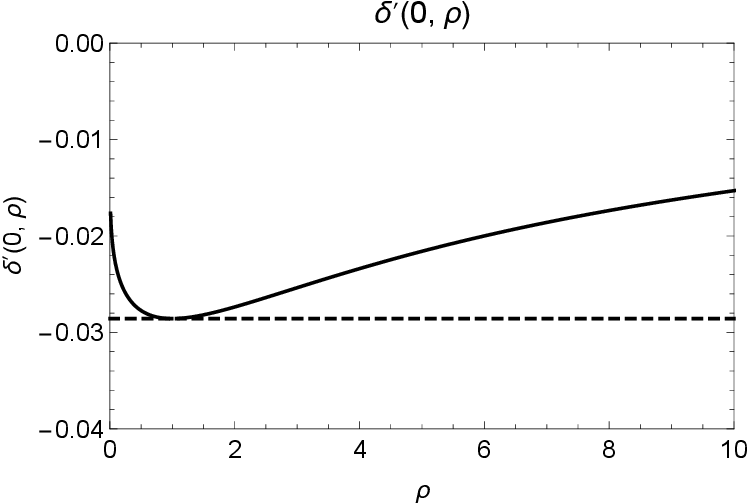}
    \caption{ Plot of $\delta'(0, \rho)$ vs $\rho$. The dashed line is at $-\frac{1}{35}$ and corresponds to the minimum value of $\delta'(0,\rho)$ which is reached at $\rho=1$.}
\label{Fig:deltaprime}
 \end{figure}
%%%%%%%%%%%%%%%%%%%%%%%%%%%%%%%%%%%%%%%

%%%%%%%%%%%%%%%%%%%%%%%%%%%%%%%%%%%%%%
As shown in Proposition \ref{prop:main}, the property (iv) yields the bound \eqref{varthetabound} on $|\vartheta(t,\rho)|$. 

\begin{remark}\label{rmk:strongerbound}
Combining the properties  (iii) and (iv) give the stronger inequality $|\delta(\tau,\rho)|
\leq \min(\frac{1}{35} \tau,1 )$, which leads to the stronger error bound
$|\vartheta(t,\rho)| \leq \vartheta_{max}(t)$
with
\begin{equation}
\vartheta_{max}(t) = \frac{1}{70} t - \sqrt{\frac{35}{t}} Ei_{-1/2}\Big(\frac{35}{t}\Big) 
+ \mbox{Erfc}\Big(\sqrt{\frac{35}{t}}\Big) \,.
\end{equation}
Here $Ei_\alpha(z) = \int_1^\infty e^{-zt} t^{-\alpha} dt$ is the exponential integral function.
For sufficiently small $t<10$, $\vartheta_{max}(t)$ is well approximated by $\frac{1}{70}t$, 
which recovers the simpler bound \eqref{varthetabound}. For larger $t$ it remains finite and approaches 1 as $t\to \infty$.
\end{remark}

%%%%%%%%%%%%%%%%%%%%%%%%%%%%%%%%%%%%%%%%%%%%%%

\end{document}